\documentclass[reqno,11pt,a4paper]{amsart}

\usepackage[all]{xy}
\usepackage{amsthm}
\usepackage{amssymb}
\usepackage{amsmath,amscd}
\usepackage{amsfonts,bbm}
\usepackage{mathrsfs}
\usepackage{threeparttable}
\usepackage{tikz}
\usepackage{lscape}
\usepackage{mathtools}
\usepackage[colorlinks=true,citecolor=blue,linkcolor=blue]{hyperref}
\usetikzlibrary {arrows.meta}
\usepackage{tikz-cd}
\usepackage{float}

\def\11{{1\kern-3.5pt 1}}
\def\mumu{{\mu\kern-4.2pt\mu}}

\def\boxtimes{\setbox0\hbox{$\Box$}\copy0\kern-\wd0\hbox{$\times$}}

\def\e{{\mathbf {e}}}
\def\f{{\operatorname {f}}}

\def\id{\operatorname {id}}

\def\rad{\operatorname {rad}}

\def\k{\mathbbm{k}}

\def\gldim{\operatorname{gldim}}

\def\mod{\operatorname{mod}}

\newtheorem{lemma}{Lemma}[section]
\newtheorem{proposition}[lemma]{Proposition}
\newtheorem{theorem}[lemma]{Theorem}
\newtheorem{corollary}[lemma]{Corollary}

{

}

\theoremstyle{definition}

\newtheorem{example}[lemma]{Example}
\newtheorem{definition}[lemma]{\sl Definition}

\newtheorem*{question*}{Question}

\theoremstyle{remark}

\newtheorem{remark}[lemma]{Remark}

\begin{document}

\pagenumbering{arabic}

\title
{\textnormal{Quasi-diagrams and gentle algebras}}

\author[Hu, Wang \& Ye]{Haigang Hu, Xiao-Chuang Wang and Yu Ye}

\address{\begin{minipage}{12.2cm}{Hu: School of Mathematical Sciences, University of Science and Technology of China, Hefei, Anhui 230026, CHINA}\end{minipage}}

\email{huhaigang@ustc.edu.cn; huhaigang\_phy@163.com} 

\address{\begin{minipage}{12.2cm}{Wang: School of Mathematical Sciences, University of Science and Technology of China, Hefei, Anhui 230026, CHINA}\end{minipage}}

\email{wxchuang@mail.ustc.edu.cn}

\address{\begin{minipage}{12.2cm}{Ye: School of Mathematical Sciences, Wu Wen-Tsun Key Laboratory of Mathematics, University of Science and Technology of China, Hefei, Anhui 230026, CHINA 
\\
and 
\\
Hefei National Laboratory, University of Science and Technology of China, Hefei 230088, CHINA}\end{minipage}}

\email{yeyu@ustc.edu.cn}




\begin{abstract} 
Any gentle algebra $A$ with one maximal path corresponds to  a unique quasi-diagram $\alpha$. 
We introduce the regularity for $\alpha$, and show that $A$ has finite global dimension if and only if $\alpha$ is regular.
We  characterize regular quasi-diagrams which remain regular under the dihedral group action. We prove that the set of maximal chord diagrams is the ``biggest'' one among the sets closed under taking Koszul dual and rotations.
\end{abstract}

\maketitle
\thispagestyle{empty}


\section{Introduction} 

Throughout $\k$ is an algebraically closed field of characteristic 0, and all vector spaces and algebras are over $\mathbbm k$. 

Gentle algebras were introduced by Assem, Happel and Skwr\'onski in the 1980s \cite{AH, AS} as a generalization of iterated tilted algebras of type $\mathbb{A} _n$ and $\tilde{\mathbb A}_n$.
In recent years, gentle algebras have attracted much attention in the representation theory of associative algebras due to their nice homological properties.
Quite remarkably, gentle algebras connect closely to many areas of mathematics, e.g., Lie algebras \cite{HK}, cluster theory \cite{BZ}, homological mirror symmetry \cite{HKK,LP}, etc.
 
It was shown that every gentle algebra $A$ can be obtained by gluing (vertices of) $\mathbb{A}_n$ quivers \cite[Section 2]{BD}, say quivers of the form
$$
\xymatrix{
\overset{1}{\circ} \ar[r] & \overset{2}{\circ} \ar[r] & \cdots \ar[r] & \overset{n}{\circ}.
}
$$
In other words, there is a canonical  radical embedding from $A$ into a product of path algebras of $\mathbb{A}_n$ quivers as explained below. For more details on radical embeddings, we refer to \cite[Section 3]{EHIS} (see also \cite{KL}).

Let $A = \k Q_A/I_A$ be a gentle algebra (Definition \ref{defn-gent}). Then every arrow of $Q_A$ belongs to a unique maximal nonzero path in $A$. Let $p_1, p_2,\dots, p_m$ be all maximal paths and $l_i$ the length of of $p_i$ for $1 \leq i \leq m$. We associate each $p_i$ a quiver $Q^i$ of type $\mathbb{A} _{l_i+1}$. Let $Q$ be the quiver with connected components $Q^1, Q^2, \cdots, Q^m$, and $R=\k Q$ the path algebra. Clearly we have an algebra isomorphism $ 
R\cong \k Q^1 \times \k Q^2 \times \cdots \times \k Q^m
$.  

Let
$
\mathcal{N}_A = \{ ( i, j ) \in \mathbb{N} \times \mathbb{N} \mid 1 \leq i \leq m, \ 1 \leq j \leq l_i + 1 \}
$ be an index set of vertices of $Q$, where $(i,j)$ refers to the $j$-th vertex of the quiver $Q^i$. Let ${e}^{i}_j$ be the idempotent of $kQ^i$ corresponding to the vertex $(i,j)$. 
We define an equivalence relation 
$
\sim 
$ 
on $\mathcal{N}_A$ as follows: 
$
(i, j) \sim  (i', j') 
$
if and only if the $j$-th vertex of the path $p_i$ coincides with the $j'$-th vertex of $p_{i'}$ in $Q_A$. By the definition of gentle algebras one can show that any equivalence class contains at most two elements.

For any $(i,j) \in \mathcal{N}_A$, we set
$ 
e_{i,j} = \sum\limits_{(u,v)\in \mathcal N_A\atop  (u,v)\sim (i,j)} e_{v}^{u} 
$.
Then $e_{i,j}^2=e_{i,j}$, $e_{i,j} = e_{u,v}$ if $(i,j)\sim (u,v)$, and $e_{i,j}e_{u,v}=0$ if $(i,j)\nsim (u,v)$. Let
\[S =  \sum_{i=1}^m\sum_{j=1}^{l_i+1} \k e_{i,j} \bigoplus \rad R =(\bigoplus_{(i,j)\in \mathcal N_A/\sim} \k e_{i,j})\bigoplus \rad R\] 
be the linear span of $e_{i,j}$'s and 
the Jacobson radical of $R$. Clearly $S$ is an subalgebra of $R$, and the inclusion $S\subseteq R$ is a {\it radical embedding} \cite[Lemma 3.1]{EHIS}, where by radical embedding it is meant that $\rad S = \rad R$.

One can show that $A$ is isomorphic to $S$ as algebras, and hence there exists a radical embedding $A\hookrightarrow R$. We say that the pair $(\mathcal{N}_A, \sim)$ is the {\it gluing datum} associated to $A$, and the gentle algebra $A$ is obtained by gluing idempotents with the gluing datum  $(\mathcal{N}_A, \sim)$. 

For a radical embedding $f\colon A \to B$, 
it is natural to compare the homological properties of $A$ and $B$ \cite{EHIS,XX}. In general, $A$ and $B$ may have totally different homological behavior and it is hard to tell what property is preserved under radical embedding. For instance, in our case,  the gentle algebra  $A$ is radically embedded in some $\k Q$. Clearly $\k Q$ is a hereditary algebra, while A may have infinite global dimension. We are interested in the naive question when $A$ has finite global dimension. 
Specifically, in this paper, we focus on the case when the gentle algebra $A$ has only one maximal path. The first question we aim to solve is:

\medskip
\noindent{\bf Question 1.}
Let $A$ be a gentle algebra obtained by gluing one $\mathbb{A} _n$ quiver $Q$ with gluing datum $(\mathcal{N}_A, \sim)$. 
Is there a simple way to determine whether $A$ has finite global dimension by studying the gluing datum $(\mathcal{N}_A, \sim)$?  
\medskip

The symmetric groups turn out to be a handy tool in studying this problem. 
Let $A$ and $(\mathcal{N}_A, \sim)$ as in the question above. Note that in this case, $\mathcal N_A$ is simply identified with $\{1,2, \cdots, n\}$, and $\sim$ is viewed as a partition of $\{1,2, \cdots, n\}$.
Let $\mathfrak{S}_n$  be the symmetric group of degree $n$. We associate a permutation $\alpha \in \frak{S}_n$ to $(\mathcal{N}_A, \sim)$ as follows,
$$
\alpha(i) = \begin{cases} j,  &\text{\quad if } \exists i\ne j\in\mathcal{N}_A, j \sim i;\\
i, &\text{\quad otherwise}. 
\end{cases}
$$
Note that $\alpha$ is an involution and hence a  {\it quasi-diagram} (a generalization of chord diagrams) in the sense of \cite{CM}.
We call $\alpha$ the quasi-diagram associated to $A$.   
This gives a one-to-one correspondence
$$
\{\text{gentle algebras with one maximal path}\}/\cong  \overset{1:1} {\longleftrightarrow}  \{\text{quasi-diagrams}\}.
$$

Consider the $n$-cycle $\zeta = (123\cdots n) \in \frak{S}_n$ and the natural actions of $\zeta\alpha$ and $\alpha\zeta$ on the set $\{1,2,\cdots, n\}$. The following result gives an answer to the above question.

\begin{theorem}[Theorem \ref{thm-main}] \label{thm-intro-1}
Let $A$ be a gentle algebra with one maximal path, and $\alpha \in \frak{S}_n$ the associated quasi-diagram. Then the following are equivalent. 
\begin{itemize}
\item[(1)] The global dimension $\gldim A  < \infty$.
\item[(2)] Any $\zeta\alpha$-orbit contains either $1$ or at least one isolated point of $\alpha$.
\item[(3)] Any $\alpha\zeta$-orbit contains either $n$ or at least one isolated point of $\alpha$.
\end{itemize}
Here by an {\it isolated point} of $\alpha$ it is meant a point fixed by $\alpha$.
\end{theorem}

We say a quasi-diagram $\alpha$ is {\it regular} if it satisfies the equivalent conditions (2), (3) in the above theorem. The theorem tells us how to determine whether $A$ has finite global dimension by checking the associated quasi-diagram, and it seems to be  relatively easy, see  Example \ref{exm-1}. 
Moreover, we also provide a  method to calculate the global dimension of $A$ in case it is finite, see  Proposition \ref{prop-gldim}.

Let $P_n$ be an $n$-gon with sides labeled by $1,2, \dots, n$ consecutively around its boundary. 
Then a quasi-diagram $\alpha \in \frak{S}_{n}$ assigns each pair of sides $i, \alpha(i)$ of  $P_n$  a chord as shown in Example \ref{exm-draw}. 
The dihedral group $D_n$ of order $2n$ (viewed as a subgroup of $\frak{S}_{n}$) acts on quasi-diagrams by conjugation:
$$
g \cdot \alpha = g \alpha g^{-1}, \ g \in D_n, \ \alpha \in \frak{S}_n. 
$$
Let $A, A'$ be gentle algebras associated to quasi-diagrams $\alpha, g \cdot \alpha \in \frak{S}_{n}$ for some $g \in D_n$. Then in general, $A$ and $A'$ may be quite ``different''. For example, it is possible that $A$ have finite global dimension while $A'$ may not, or in other words, the regularity  of quasi-diagrams may not be preserved under conjugation. 
We are interested in the following question:

\medskip
\noindent{\bf Question 2. } \label{que-2}
Let $\alpha \in \frak{S}_n$ be a regular quasi-diagram. When is $g \cdot \alpha$ regular for all $g \in D_n$?
\medskip

A quasi-diagram $\alpha$ is said to be {\it rotatably regular} if $\zeta^l \cdot \alpha = \zeta^l \alpha \zeta^{-l} $ is regular for any integer $l$. The following result provides several equivalence conditions for $\alpha$ being rotatably regular, answering Question 2 to some extend.

\begin{theorem} {\rm (Proposition \ref{prop-ref}, Theorem \ref{thm-rot})} \label{thm-intro-2}
Let $\alpha \in \frak{S}_n$ be a quasi-diagram. Then the following statements are equivalent.
\begin{enumerate}
	\item $g \cdot \alpha$ is regular for all $g \in D_n$.
	\item $\alpha$ is rotatably regular.
	\item Either $\alpha $ is maximal, or each orbit of $\zeta \alpha$ contains at least one isolated point of $\alpha$.
\end{enumerate}
\end{theorem}

We introduce the notions of expansion and contraction for quasi-diagrams, and compare the faces, isolated points and regularity of a quasi-diagram with the ones of its expansions and contractions, see Section 5 for detail.  Recall that a chord diagram is a quasi-diagram without isolated points, and a quasi-diagram $\alpha$ is maximal if $\zeta\alpha$ has only one orbit, see Definition \ref{defn-dia}. The following connects quasi-diagrams and chord diagrams.

\begin{proposition} {\rm (Proposition \ref{prop-mdisptext})}
	Every nontrivial quasi-diagram is an iterated expansion of a chord diagram, and every nontrivial maximal quasi-diagram is an iterated expansion of a maximal chord diagram. 
\end{proposition}
 
 The proposition provides a possible way to restrict problems to the special case of chord diagrams when working with quasi-diagrams. For instance, we apply it to reattain a counting formula of maximal quasi-diagrams by using the one of maximal chord diagrams, see Proposition \ref{prop-enum} (2).

In the last part of the paper, we discuss maximal chord diagrams. We show in Proposition \ref{prop-mcd} several equivalence characterization for maximal chord diagrams by using Koszul dual and the conjugate action of the dihedral group. The subset of $\frak{M}_n$ ($n \geq 3$) consisting of maximal chord diagrams turns out to the ``biggest'' subset closed under taking rotations and Koszul dual, see Proposition \ref{prop-setmcd}.  Moreover, by a recent result of Chang and Schroll \cite{CS}, maximal chord diagrams exactly correspond to those gentle algebras $A$, such that $A$ has finite global dimension and  $D^b(\mod A)$ has no full exceptional sequences, see Remark \ref{rem-tg-mc}.

\subsection*{Acknowledgments}
This work is partially supported by the National Natural Science Foundation of China (Grant Nos. 12131015,  12161141001 and 12371042) and the Innovation Program for Quantum Science and Technology (Grant No. 2021ZD0302902). 

\section{Preliminaries}

In this paper all algebras considered are basic. A path in a quiver $Q$ is a sequence $a_1a_2\cdots a_l$ (composed from left to right) of arrows with $t(a_i)=s(a_{i+1})$ for all $i=1, 2, \cdots, l-1$, where for each $i$, $s(a_i)$ and $t(a_i)$ denote the source and target of $a_i$ respectively, $l$ is called the \emph{length} of the path. We use $Q_l$ to denote the set of paths of length $l$ in $Q$, in particular, $Q_0$ is the set of vertices which are identified with trivial paths, and $Q_1$ is the set of arrows. Maps are composed from right to left, that is the composite of $f\colon  X \to Y$ and $g\colon Y \to Z$ is denoted by $g\circ f\colon X \to Z$.

\begin{definition} \label{defn-gent}
 An algebra $A$ is called a {\it locally gentle algebra} if it is isomorphic to $\k Q/I$, where
\begin{itemize}
\item[(1)] $Q$ is a finite quiver, and for every vertex $i\in Q_0$, there are at most two arrows ending at $i$
 and at most two arrows starting at $i$;
\item[(2)] $I$ is generated by paths of length two;
\item[(3)] for every arrow $a \in Q_1$, there is at most one arrow $b $ such that 
 $ab \in  I$, and at most  one arrow $c $ such that
$ca  \in I$;
 and there is at most one arrow $b' $ such that 
 $ab' \notin I$, and at most one arrow $c' $ such that 
 $c'a \notin I$.
\end{itemize}
A locally gentle algebra is called a {\it gentle algebra } if it is finite dimensional. 
\end{definition}

\begin{remark} \label{rem-maxpath}
	A (nonzero) path in $A=\k Q/I$ means a path $p$ in $Q$ such that $p\not\in I$, i.e., $p\ne 0$ in $A$. By definition, for a locally gentle algebra, each arrow $a_0$ either appears in a unique maximal path $\cdots a_{-1}a_0a_1\cdots$, or there exists an oriented cycle $a_0a_1\cdots a_ra_0$ in $A$. Recall that a path $p$ is maximal in $A$ if $pa=ap=0$ for any arrow $a$. 
\end{remark}
Note that the permutation $\alpha \in \frak{S}_n$ associated to a gentle algebra $A$ mentioned in the introduction is an involution, i.e., $\alpha^2 = \id$. 
Thus the cycles of $\alpha$ is of length $2$ or  $1$. The following definitions  mimic those in \cite{CS}. 

\begin{definition} \label{defn-dia}
Let $\zeta_n = (12 \dots n), \alpha \in \frak{S}_n$.
\begin{itemize}
\item[(1)] We call $\alpha$ a {\it quasi-diagram} if $\alpha$ is an involution. The identity $\id \in \frak{S}_n$ is called the {\it trivial quasi-diagram}. 

\noindent Now let  $\alpha \in \frak{S}_n$ be a quasi-diagram.
\item[(2)] A cycle of length $2$ of  $\alpha$ is called a {\it chord} of $\alpha$, a point fixed by $\alpha$ is called an {\it isolated point} of $\alpha$. 
\item[(3)] We call $\alpha$ a {\it chord diagram} if it is isolated point free. 
\item[(4)] Write $\zeta_n \alpha= w_1\cdots w_r$ as a complete product of disjoint cycles, where complete means that each $i$ occurs in some cycle  in the product, and we will distinguish between the 1-cycles $(i)$ and $(j)$ for $i\neq j$, although they are both the identity mapping viewed as permutations. Then each $w_i$ is called a {\it face} of $\alpha$. By abuse of notations, we will also call an orbit of $\zeta_n \alpha$ a face. 
\item[(5)] We say that $\alpha$ is {\it maximal} if it has only one face. 
\end{itemize}
We use $\frak{D}_n$ and $\mathfrak M_n$ to denote the set of quasi-diagrams and the subset of maximal chord diagrams in $\frak{S}_n$ respectively.
\end{definition}

For the rest of the paper, 
$\zeta_n$  always  denotes the $n$-cycle
$
(12\cdots n)\in \frak{S}_n
$, 
which is also simply denoted by $\zeta$ when there is no confusion on $n$.

\begin{example}
Let $\alpha = (12)(45) \in \frak{D}_5$. Then $(13), (45)$ are chords of $\alpha$, and $3$ is an isolated point of $\alpha$. Since $\zeta_5\alpha = (134)(2)(5)$, $\alpha$ has three faces: $(134),(2),(5)$. 
\end{example}

\begin{remark} \label{rmk-1}
For every gentle algebra $A$, we can associate to $A$ a {\it marked surface} $\mathbb{S}_A$, 
which is an oriented surface with boundary so that the ribbon graph of $A$ can be (filling) 
embedded \cite{FL,OPS}. The notion of marked surfaces is shown
to be very useful in the study of $D^b(\mod A)$, 
the bounded derived category  
of the category $\mod A$ consisting of finitely generated right $A$-modules \cite{CS,LSV,OPS}.

Let $A$ be a gentle algebra with one maximal path, and  $\alpha$ be the associated quasi-diagram. Let $\widehat{\mathbb{S}_A}$ be the surface without boundary obtained from $\mathbb{S}_A$ by gluing an open disc to each of the boundary components of $\mathbb{S}_A$.  It is not hard to check that 
\begin{align*}
 V :=&  \#\{\text{vertices  in }\widehat{\mathbb{S}_A}\} = \#\{\text{isolated points of }\alpha\} +1, \\
 E :=&  \#\{\text{edges  in }\widehat{\mathbb{S}_A}\} = \#\{\text{isolated points of }\alpha\} +\#\{\text{chords of }\alpha\}, 
\end{align*}
and 
$$
\{\text{faces  in }\widehat{\mathbb{S}_A}\} \overset{1:1}{\longleftrightarrow}\{\text{faces of }\alpha\}
$$
(which explains the name faces of $\alpha$). 

Let $F := \#\{\text{faces  in }\widehat{\mathbb{S}_A}\}$.
By the Euler characteristic formula $2-2g = \chi = V-E+F$, we get the following formula for genus of $\widehat{\mathbb{S}_A}$ (as well as $\mathbb{S}_A$)
\begin{equation} \label{equ-gen}
2g = \#\{\text{chords of }\alpha\} -  \#\{\text{faces of } \alpha\} +1. 
\end{equation}
It implies that  $\mathbb{S}_A$ has maximal possible genus if $\alpha$ is maximal. 
\end{remark}

We recall the following notion of Koszul algebras, which are an important class of graded algebras with nice homological properties. By definition, a locally gentle algebra is quadratic monomial and hence a Koszul algebra.

\begin{definition}[{\cite[Definition 1.2.1]{BVW}}]
A positively graded algebra $A = \bigoplus_{i \geq 0} A_i$ is called {\it Koszul} if $A_0$ is semisimple and if the graded right $A$-module $A_0$ admits a graded projective resolution 
$$
\xymatrix{
\cdots \ar[r] &P^i \ar[r] &\cdots \ar[r]  & P^1 \ar[r] & P^0 \ar[r] & A_0 \ar[r] & 0 
}
$$
such that each $P^i$ is generated by its component in degree $i$, i.e., $P^i = P^{i}_iA$.
\end{definition}

The Koszul dual $A^!$ of a Koszul algebra $A$ is defined to be the Yoneda algebra $\operatorname{Ext}_A(A_0,A_0)$ \cite{BVW}. It is well known that $A^!$ is again a Koszul algebra, and $(A^!)^!\cong A$ as graded algebras.

\begin{proposition}[{\cite[Proposition 3.4]{BH}}] \label{prop-kos}
A locally gentle algebra $A = \k Q/I$ (with grading given by path lengths) is Koszul. The Koszul dual $A^!$ of $A$ is isomorphic to the locally gentle algebra $\k Q^!/I^!$ where
\begin{itemize}
\item[(1)] the quiver $Q^!$ is equal to the opposite quiver of $Q$, that is, the quiver obtained from $Q$ by reversing all arrows;
\item[(2)] the ideal $I^!$ is generated by the opposites of the paths $p$ of length two in $Q$ which do not appear in $I$.
\end{itemize}
\end{proposition}

Thus for a locally gentle algebra $A = \k Q/I$, a path with full relations
$$
\begin{tikzpicture}
\draw (-2,0) circle[radius= 0.17 em];
\draw (-0.7,0) circle[radius= 0.17 em];
\draw (0.4,0) node[right] {$\cdots$};
\draw (2.25,0) circle[radius= 0.17 em];
\draw (3.35,0) node[right] {$\cdots$};
\draw (5.2,0) circle[radius= 0.17 em];

\draw (1.6,0) node[below] {$a_i$};
\draw (2.9,0) node[below] {$a_{i+1}$};

\draw[->] (-1.85, 0) --  (-0.85,0) ; 
\draw[->] (-0.55, 0) --  (0.45,0) ; 
\draw[->] (1.1, 0) --  (2.1,0) ; 
\draw[->] (1.1, 0) --  (2.1,0) ; 
\draw[->] (2.4, 0) --  (3.4,0) ; 
\draw[->] (4.05, 0) --  (5.05,0) ; 

\draw[densely dotted] (-1, 0) arc (180:0:0.3);
\draw[densely dotted] (1.945, 0) arc (180:0:0.3);
\end{tikzpicture}
$$
 will become a (nonzero) path in the Koszul dual $A^! = \k Q^!/I^!$. Here a \emph{path with full relations} in $A$ means a path in the quiver $Q$ such that $a_ia_{i+1} \in I$ for any two consecutive arrows $a_i, a_{i+1}$ in the path.

Then we obtain a characterization of the global dimension of $A$, which is a special case of the following well-known result (\cite{AH}).
\begin{lemma}\label{lem-gldim}
	Let $A=\k Q/I$ be a quadratic monomial algebra, where $Q$ is a finite quiver, and $I$ is an ideal generated by a set of paths of length 2. Then
 $\gldim A$  equals the maximal length of paths with full relations in $A$. 
\end{lemma}

\begin{corollary} \label{cor-kos}
Let $A$ be a gentle algebra with one maximal path, and $\alpha \in \frak{D}_n$ the associated quasi-diagram. 
Then the Koszul dual $A^!$ has only one maximal path if and only if $\gldim A = n-1$.
\end{corollary}
\begin{proof}
	Note that the quivers $Q$ and $Q^!$ have exactly $n-1$  arrows. Moreover, since $A^!$ is also locally gentle, either $A^!$ has finite dimension and each arrow of $Q^!$ appears in a unique maximal path, or $A^!$ has infinite dimension and contains a (nonzero) oriented cycle. Now the corollary follows easily from Lemma \ref{lem-gldim}, which says that $\gldim A$ equals the maximal length of nonzero paths in $A^!$, or the maximal length of paths with full relations in $A$.
\end{proof}

\begin{remark}\label{rem-kosgldim}
	Let $\alpha \in \frak{D}_n$ be a quasi-diagram and $A$ be the gentle algebra. We will say that the Koszul dual of $\alpha$ exists if the Koszul dual $A^!$ of $A$ also has one maximal path, and in this case, the quasi-diagram $\alpha^!$ associated to $A^!$, denoted by quasi-diagram $\alpha^!$, is called the {\it Koszul dual quasi-diagram} of $\alpha$. Note that Proposition \ref{prop-kos} also implies that $\gldim A^! = n-1$. 
\end{remark}

\begin{example} \label{exm-id2}
Let $\alpha = \id \in \frak{D}_2$, and $A$ the associated gentle algebra. Then $A$ is the path algebra of the quiver
$$
\xymatrix{
\circ \ar[r] & \circ
}.
$$
Thus $\gldim A = 1$. By Corollary \ref{cor-kos}, the Koszul dual $\alpha^!$ exists, and we can check that $\alpha^! = \alpha$. 
\end{example}

\section{Global Dimension} \label{sec-pro}

In this section, we prove the first main result mentioned in the introduction. We begin with an easy lemma. 

\begin{lemma} \label{lem-1} 
Let $\alpha \in \frak{D}_n$ be a quasi-diagram and $\zeta = (123\cdots n)$. Then
\begin{itemize} 
\item[(1)]  The elements $1$ and  $\alpha(n)$ have the same $\zeta \alpha$-orbit;
\item[(2)]  The elements $n$ and  $\alpha(1)$ have the same $\alpha \zeta$-orbit. 
\end{itemize}
\end{lemma}

\begin{proof}
(1) follows from 
$\zeta \alpha (\alpha (n)) = \zeta \alpha^2 \zeta^{-1}(1) = \zeta(n) = 1$,
since $\alpha^2 = \id$. 

(2) follows from $\alpha \zeta  (n) = \alpha (1)$. 
\end{proof}

Let $A = \k Q_A/I_A$ be the gentle algebra associated to $\alpha$. Note that the global dimension of $A$ is equal to the maximal number $m$ such that there is a path of length $m$ with full relations in $A$.

Before state our main result, we introduce two special classes of orbits of the set $\{1,2,\cdots, n\}$ under the natural actions of $\zeta\alpha$ and $\alpha\zeta$, say
$$
\mathscr{A}_{\alpha} := \{ \text{orbits of } \zeta \alpha \text{ containing no isolated points of }  \alpha  \text{ nor } 1\},
$$
and
$$
\mathscr{B}_{\alpha} := \{ \text{orbits of } \alpha \zeta \text{ containing no isolated points of }  \alpha  \text{ nor } n\}.
$$
The following result is based on the key observation that
a path with full relations in $A$ corresponds to 
the consecutive non-isolated points in a face of $\alpha$ (i.e., an orbit of $\zeta\alpha$).  

\begin{theorem} \label{thm-main}
Let $A$ be a gentle algebra with one maximal path, and $\alpha \in \frak{D}_n$ the associated quasi-diagram. Then the following are equivalent. 
\begin{itemize}
\item[(1)] $\gldim A  < \infty$.
\item[(2)] $\#(\mathscr{A}_{\alpha}) = 0$. 
\item[(3)] $\#(\mathscr{B}_{\alpha}) = 0$. 
\end{itemize}
\end{theorem}

\begin{proof}
By assumption $A = \k Q_A/I_A$ is a gentle algebra obtained by gluing one $\mathbb{A} _n$  quiver $Q$. 
Let $\alpha \in \frak{D}_n$ be the associated quasi-diagram. 

{ $(2) \Leftrightarrow (3)$.} Clearly  $\alpha\zeta= \alpha(\zeta\alpha)\alpha^{-1}$. Hence a subset $\omega$ is an orbit of $\zeta\alpha$ if and only if $\alpha(\omega)=\{\alpha(i)\mid i\in \omega\}$ is an orbit of $\alpha\zeta$. 
By Lemma \ref{lem-1},  $1\in \omega$ if and only if $\alpha(n)\in \omega$, if and only if $n\in\alpha^{-1}(\omega) = \alpha(\omega)$. Moreover, $\omega$ contains a fixed point $i$ of $\alpha$ if and only if $\alpha(\omega)$ does. Thus we have shown that for any $w\in\mathscr{A}_\alpha$ if and only if $\alpha(\omega)\in \mathscr{B}_\alpha$, and the equivalence of (2) and (3) follows.

{$(1) \Leftrightarrow (3)$}. As mentioned in the introduction, $A$ is a subalgebra of $\k Q$, where $Q$ is the quiver
\[
\xymatrix{
	\overset{1}{\circ} \ar^-{a_1}[r] & \overset{2}{\circ} \ar^-{a_2}[r] & \cdots \ar^-{a_{n-1}}[r] & \overset{n}{\circ}.
}
\]
Let ${e}_i$ denote the trivial path at the vertex $i$ in $Q$. Clearly $e_1, e_2, \cdots, e_n$ form a complete set of orthogonal idempotents.

Consider the associated quasi-diagram $\alpha\in \mathfrak S_n$. Since $\alpha$ is an involution, we may write
$$
\alpha = (x_1,y_1) \cdots (x_u,y_u)(x_{u+1})\cdots (x_v)
$$
as a product of disjoint 2-cycles and 1-cycles. 
Then a complete set of orthogonal idempotents $\e_{1}, \dots, \e_{t}$  of $A$ is given
by 
$$
\e_i = \begin{cases} {e}_{x_i} + {e}_{y_i} & 1\le i\le u, \\ {e}_{x_i} & u+1\le i\le v, \end{cases}
$$
and $A= \bigoplus_{i=1}^t \k \e_i \bigoplus \rad \k Q$.

For any $1\le i\le n$, there exists a unique $1\le \f(i)\le t$, such that $i= x_{\mathrm f(i)}$ or $y_{\f(i)}$. Thus $\f$ defines a unique partition function
\[\f\colon \{1, 2,\cdots, n\}\to \{1,2,\cdots, t\}.\]
Note that $\f(i)=\f(j)$ if and only if $i=j$ or $i=\alpha(j)$ 

Then the quiver $Q_A$ is exactly the quiver with vertices $\{1, 2, \cdots, t\}$, arrows $\{a_1, a_2, \cdots, a_{n-1}\}$, and the source and target maps given by $s(a_i)=\f(i)$, and $t(a_i)=\f(i+1)$ for $i=1, 2, \cdots, n-1$. The defining relations of $A$ are 
$$\{a_ia_j\mid 1\le i, j\le n-1, \f(i+1)=\f(j), j\ne i+1\}.$$
Or in other words, a path $a_ia_j\in I_A$ if and only if $j = \alpha(i+1) = \alpha\zeta(i)$ and $j$ is not an isolated point of $\alpha$. 

By Lemma \ref{lem-gldim}, $A$ has infinite global dimension if and only if 
there exists an oriented cycle $a_{i_1}a_{i_2}\cdots a_{i_r}a_{i_1}$  with full relations (Figure \ref{cyc})
\begin{figure}[h]
	\centering
	\begin{tikzpicture}[x=20, y=20, baseline=(current bounding box)]
		\draw[->] (-1.9, 0.173) --  (-1.08,1.5916) ; 
		\draw[->] (-0.83,1.73) -- (0.83, 1.73) ; 
		\draw[->] (1.1, 1.557) -- (1.93, 0.1211) ; 
		\draw[->] (1.91, -0.173) -- (1.07, -1.6089) ; 
		\draw[->] (0.83, -1.73) -- (-0.83, -1.73) ; 
		\draw[dashed,->] (-1.1, -1.557) -- (-1.9,-0.173) ; 
		\draw (-2,0) circle[radius= 0.16 em];
		\draw (-1,1.73) circle[radius= 0.16 em];
		\draw (1, 1.73) circle[radius= 0.16 em];
		\draw (2, 0) circle[radius= 0.16 em];
		\draw (1, -1.73) circle[radius= 0.16 em];
		\draw (-1, -1.73) circle[radius= 0.16 em];
		\draw (-2,0) node[left] {$f(i_r)$};
		\draw (-1.5,1) node[left] {$a_{i_r}$};
		\draw (2,0) node[right] {$f(i_3)$};
		\draw (0,1.75) node[above] {$a_{i_1}$};
		\draw (-1.5,1.73) node[above] {$f(i_1)$};
		\draw (1.45,1) node[right] {$a_{i_2}$};
		\draw (1.5, 1.73) node[above] {$f(i_2)$};
		\draw[densely dotted] (-0.5, 1.73) arc (-5:-115:0.5);
		\draw[densely dotted] (0.5, 1.73) arc (-175:-65:0.5);
		\draw[densely dotted] (-0.5, -1.73) arc (0:120:0.5);
		\draw[densely dotted] (0.5, -1.73) arc (180:60:0.5);
		\draw[densely dotted] (1.72, 0.46) arc (115:245:0.5);
		\draw[densely dotted] (-1.74, 0.46) arc (60:-60:0.5);
	\end{tikzpicture}
	\caption{The oriented cycle with full relations} 
	\label{cyc}
\end{figure}
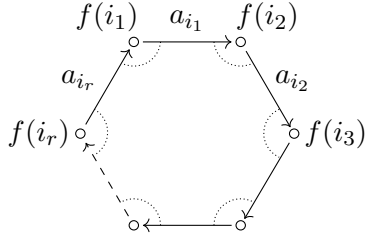
if and only if the subset $\omega=\{i_1, i_2, \cdots, i_r\}$ forms an $\alpha\zeta$-orbit which 
contains no isolated points of $\alpha$ nor $n$, i.e. $\omega\in \mathscr B_\alpha$ (Figure \ref{fig-2}).
\begin{figure}[h]
\small
$$
\xymatrix{
  & i_1 \ar[r]^-{\zeta}  \ar[r] &  i_1+1  \ar[d]^-{\alpha}&  \\
 i_r \ar[r]^-{\zeta} & i_r+1 \ar[u]^-{\alpha}  & i_2 \ar[r]^-{\zeta}  & i_2+1 \ar[d]^-{\alpha}  \\
 i_{r-1}+1 \ar[u]^-{\alpha} & \cdot \ar@{-->}[l] &  \cdot \ar@{-->}[d] & i_3 \ar@{--}[l] \\
 & i_s+1 \ar@{--}[u] & i_s \ar[l]_-{\zeta} &
}
$$
	\caption{The $\alpha\zeta$-orbit $w$} 
	\label{fig-2}
\end{figure}
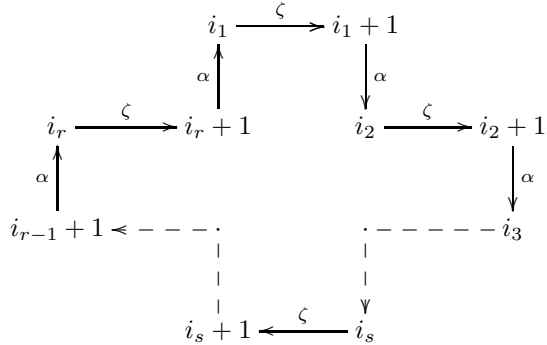
\end{proof}

We call a quasi-diagram $\alpha$ {\it regular} if it satisfies the equivalent conditions (2) and (3) in above theorem.  Then we draw the following consequence.

\begin{corollary}
Maximal quasi-diagrams are regular. 
\end{corollary}

\begin{example} \label{exm-1}
Consider the quiver 
$$
\xymatrix{
Q = \overset{1}{\circ} \ar[r] & \overset{2}{\circ} \ar[r] &  \overset{3}{\circ} \ar[r] & \overset{4}{\circ}.
}
$$
Let $A$, and $A'$ be gentle algebras obtained by gluing $Q$ with quasi-diagrams
$$\alpha = (13)(24), \text{ and } \alpha' = (12) \in \frak{S}_4$$
respectively.  Then we have
\begin{align*}
\zeta_4 \alpha = (1432), \ \zeta_4 \alpha' = (134)(2). 
\end{align*}
Thus $\#(\mathscr{A}_\alpha) = 0$, and $\#(\mathscr{A}_{\alpha'}) = 1$. By Theorem \ref{thm-main}, we have $\gldim A < \infty$, and $\gldim A' = \infty$. 
\end{example}

For $j \in \{1, \dots, n-1 \}$, let $g_j$ be the smallest positive integer such that $(\alpha \zeta )^{g_j} (j)=n$ or an isolated point of $\alpha$, and let $g_n = 0$.
Similarly, for $i \in \{1, \dots , n\} \setminus \{\alpha(n) \}$, let $d_i$ be the smallest positive integer such that $(\zeta \alpha)^{d_i} (i)$ is an isolated point of $\alpha$ or $(\zeta \alpha)^{d_i} (i) = \alpha(n)$, and let $d_{\alpha(n)} = 0$.

\begin{proposition} \label{prop-gldim}
Let $\alpha \in \frak{D}_n$ be a regular quasi-diagram, and $A$ the associated gentle algebra. For $i , j \in \{1, \dots, n \}$, let $d_i,g_j$ as above. Then
\begin{align*}
\gldim A &= \max\{ g_j \mid j \text{ is an isolated point of } \alpha \text{ or } j = \alpha(1) \}   \\
&= \max\{d_i \mid i \text{ is an isolated point of } \alpha \text{ or } i = 1 \}.
\end{align*}
\end{proposition}

\begin{proof}
We will show the first equality in detail and the second one can be proved similarly. See also the proof of $(2) \Leftrightarrow (3)$ in Theorem \ref{thm-main}. 

Let $A$ be a gentle algebra with one maximal path and $\alpha \in \frak{D}_n$ the associated quasi-diagram. As in the proof of Theorem \ref{thm-main}, for $i=1,\cdots,n-1$, let $a_i$ be the arrow in $\mathbb{A}_n$ quiver $Q$ starting at the vertex $i$ and ending at the vertex $i+1$. Recall that a path $a_{i_1}a_{i_2}\cdots a_{i_r}$ has full relations in $A$ if and only if $\alpha\zeta (i_u)= i_{u+1}$ for $u=1, \cdots, r-1$, and $i_2, \cdots, i_r$ are non-isolated points of $\alpha$. Let $a_{i_1}a_{i_2}\cdots a_{i_r}$ be a maximal  path with full relations in $A$. Then $i_1$ is either an isolated point of $\alpha$ or equals $\alpha(1)$. Otherwise we have $i_0=\zeta^{-1}\alpha(i_1)\ne n$ and $a_{i_0}a_{i_1}\cdots a_{i_r}$ is a path with full relations in $A$, a contradiction. Similarly, $\alpha\zeta(i_r) = n$ or an isolated point. Now the desired equality follows from Lemma \ref{lem-gldim}.
\end{proof}

\begin{example}
Consider the gentle algebra $A$ defined by the regular chord diagram $\alpha = (13)(24) \in \frak{D}_4$ in Example \ref{exm-1}. We have $d_1 = 3$ and there is no isolated point, so that $\gldim A = 3$. 
\end{example}

The following result mainly follows from Proposition \ref{prop-gldim}. It tells us when the Koszul dual $A^!$ has only one maximal path, or  equivalently, when the Koszul dual quasi-diagram is well defined (Corollary \ref{cor-kos}).  

\begin{theorem} \label{thm-kosdual}
Let $A$ be a gentle algebra with one maximal path, and $\alpha \in \frak{D}_n$ the associated quasi-diagram. Then the following are equivalent.
\begin{enumerate}
	\item The Koszul dual quasi-diagram $\alpha^!$ of $\alpha$ exists.
	\item The global dimension of $A$ is $n-1$.
	\item The quasi-diagram $\alpha$ is one of the following types:
		\begin{enumerate}
		\item [Type A:]$\alpha$ is a maximal chord diagram;
		\item [Type B:]$\alpha$ is a maximal quasi-diagram with isolated points $1$ or $n$, or both;
		\item [Type C:]$(1,n)$ is a chord of $\alpha$, and $\alpha$ has exactly one isolated point and exactly two faces.
	\end{enumerate}
\end{enumerate}
Moreover, the Koszul dual $\alpha^!$ has the same type as $\alpha$ if exists. 
\end{theorem}

\begin{proof}
	{ $(1) \Leftrightarrow (2)$} follows from Corollary \ref{cor-kos} and it suffices to prove  $(2) \Leftrightarrow (3)$.
	
	{ $(3) \Rightarrow (2)$.} If (3) holds,  then we can check that $\gldim A=n-1$ case by case by applying Proposition \ref{prop-gldim}. 

	$(2) \Rightarrow (3)$. If (2) holds, then by Proposition \ref{prop-gldim}, we know that $\alpha$ has at most two faces. Otherwise, we have $d_1,d_i < n-1$ for any isolated point $i$.  There are two cases.

  Case 1:  $\alpha$ is maximal, i.e., $\alpha$ has only one face. We may write 
$
\zeta \alpha  = (1, \dots, \alpha(n))
$
since $\zeta\alpha(\alpha(n)) = \zeta(n) = 1$. Then $(\zeta\alpha)^{n-1}(1)= \alpha(n)$, and hence $d_1 \leq n-1$. For any $i\notin \{1, \alpha(n)\}$, we have $i=(\zeta\alpha)^{\sigma_i}(1)$ for some $1\le \sigma_i\le n-2$, then $(\zeta\alpha)^{n-1-\sigma_i}(i)=\alpha(n)$ and $d_i\le n-\sigma_i-1 < n-1$. Moreover, $i$ is not an isolated point of $\alpha$, otherwise, $d_1\le \sigma_i\le n-2$, and by Proposition \ref{prop-gldim} we have $\gldim A\le n-2$, which leads to a contradiction.

Now we have shown that in case $\alpha$ is maximal, $d_1 = n-1$ and any isolated point of $\alpha$ is either $1$ or $\alpha(n)$, that is, $\alpha$ is either of Type $A$ or Type $B$.

Case 2: $\alpha$ has exactly two faces. Write $\zeta\alpha = w_1 w_2$, where $w_1, w_2$ are faces of $\alpha$. We may assume $1 \in w_1$ without loss of generality. By Theorem \ref{thm-main}, there is an isolated point $i$ such that $i \neq 1$ and $i \in w_2$. Let $l_1,l_2$ be the lengths of $w_1, w_2$ respectively. Then
$$
d_1 \leq l_1 - 1 = n- l_2 - 1 \leq n-2, 
$$
where the first inequality follows from $\alpha(n) \in w_1$ (Lemma \ref{lem-1}), and 
$$
d_i \leq l_2 = n- l_1 \leq n-1.
$$  
By similar discussion as in case (i), we have $\gldim A = n-1$ if and only if $d_i = n-1 = l_2$, if and only if there is no isolated point $j \in w_2$ other than $i$. On the other hand, since $l_1 = n-l_2 = 1$ and $1 \in w_1$, we have $w_1 = (1)$, and it follows that $(1,n)$ is a chord of $\alpha$. 
Therefore $A$ is of Type C. 

We are left to prove the last statement. Assume that the Koszul dual diagram $\alpha^!$ of $\alpha$ exists. It suffices to prove the statement for $\alpha$ of Type B and Type C, and then the case of Type A holds automatically. 

First assume that $\alpha$ is of Type B. It suffices to show that $1$ is an isolated point of $\alpha$ if and only if $n$ is an isolated point of $\alpha^!$. Let $\alpha$ be a maximal quasi-diagram with isolated point $1$.  Then the isolated point $1$ corresponds to a vertex $v$ in the quiver $Q_{A}$, such that there is only one arrow $a$ with $s(a) = v$, and there is no arrow $b$ with $t(b) = v$. By Proposition \ref{prop-kos}, it implies that there is a vertex $v'$ in $Q_{A^!}$, such that there is only one arrow $a'$ with $t(a') = v'$, and there is no arrow $b'$ with  $s(b') = v'$. The existence of such an vertex $v'$ implies that $n$ is an isolated point of $\alpha^!$.

Now assume that $\alpha$ is of Type C. Since $(1,n)$ is a chord of $\alpha$. There is a path with relations in $A$ as in the following 
$$
\begin{tikzpicture}
\draw (0.4,0) node[right] {$\cdots$};
\draw (2.25,0) circle[radius= 0.17 em];
\draw (3.35,0) node[right] {$\cdots$.};

\draw (2.25,0) node[below] {$v$};

\draw[->] (1.1, 0) --  (2.1,0) ; 
\draw[->] (1.1, 0) --  (2.1,0) ; 
\draw[->] (2.4, 0) --  (3.4,0) ; 

\draw[densely dotted] (1.945, 0) arc (180:0:0.3);
\end{tikzpicture}
$$
Where $v$ corresponds to the vertex in $Q_A$ by gluing $1$ and $n$, so that there is only one arrow $a$ with $t(a) = v$ and one arrow $b$ with $s(b) = v$. By Proposition \ref{prop-kos}, there is a path in $Q_{A^!}$ as in the following
$$
\begin{tikzpicture}
\draw (0.4,0) node[right] {$\cdots$};
\draw (2.25,0) circle[radius= 0.17 em];
\draw (3.35,0) node[right] {$\cdots$};

\draw (2.25,0) node[below] {$v'$};

\draw[<-] (1.1, 0) --  (2.1,0) ; 
\draw[<-] (1.1, 0) --  (2.1,0) ; 
\draw[<-] (2.4, 0) --  (3.4,0) ; 

\end{tikzpicture}
$$
so that there is only one arrow $a'$ with $s(a') = v'$ and  one arrow $b'$ with $t(b') = v'$.
Thus there is an isolated point $i$ of $\alpha^!$ corresponds to the vertex $v'$, and clearly that $i \neq 1,n$. Thus $\alpha^!$ is of Type C. 
\end{proof}

\begin{example}
Let $A_1,A_2, A_3$ be the gentle algebras defined by the quasi-diagrams 
$$
\alpha_1 = (13)(28)(46)(57) \in \frak{S}_8, \ \alpha_2 = (13)(24) \in \frak{S}_5, \ \alpha_3 = (17)(24)(35)  \in \frak{S}_7
$$
respectively.
Then we have
$$
\zeta _8\alpha_1=(14765832), \ \zeta_5 \alpha_2 = (14325), \ \zeta _7\alpha_3 =(1)(254367).
$$
By Theorem \ref{thm-kosdual}, we have $\gldim A_1=7,\gldim A_2=4, \gldim A_3 = 6$.
Their Koszul dual quasi-diagrams are
$$
(\alpha_1)^! = (13)(28)(46)(57) = \alpha_1, \ (\alpha_2)^! = (24)(35), \ (\alpha_3)^! = (17)(35)(46).
$$
Moreover, $\alpha_1, \alpha_2, \alpha_3$ are of Type A, B, C respectively. 
\end{example}

\section{Dihedral Group Action} \label{sec-dih}

For a quasi-diagram $\alpha \in \frak{D}_{n}$, we can image $\alpha$ describes a drawing on an $n$-gon $P_n$ with its sides labeled by $1,2, \dots, n$ consecutively around its boundary: 
draw a chord between sides $i$ and $j$ for every chord $(i,j)$ of $\alpha$.
  
\begin{example} \label{exm-draw}
Let $\alpha = (13)(24), \alpha' = (14) \in \frak{D}_4$. We give the corresponding drawings in Figure \ref{P2}.
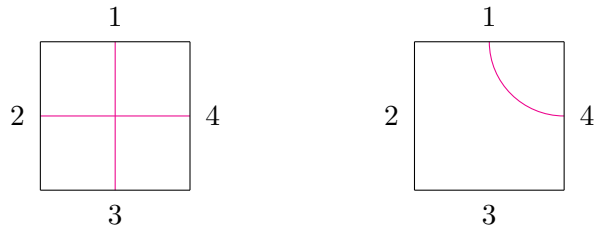
\begin{figure}[h]
  \centering
 \begin{tikzpicture}[x=20, y=20, baseline=(current bounding box)]
  \draw[magenta,-] (-6.4, 0) --  (-3.6,0) ; 
 \draw[magenta,-] (-5, 1.4) --  (-5,-1.4) ; 
 \draw[-] (-6.4, -1.4) --  (-6.4,1.4) ; 
 \draw[-] (-6.4,1.4) -- (-3.6, 1.4) ; 
 \draw[-] (-3.6, 1.4) -- (-3.6, -1.4) ; 
 \draw[-] (-3.6, -1.4) -- (-6.4, -1.4) ; 
 \draw (-5,1.5) node[above] {$1$};
 \draw (-5,-1.5) node[below] {$3$};
 \draw (-6.5,0) node[left] {$2$};
 \draw (-3.5,0) node[right] {$4$};

 \draw[magenta] (2,1.4) arc(180:270:1.4);
 \draw[-] (0.6, -1.4) --  (0.6,1.4) ; 
 \draw[-] (0.6,1.4) -- (3.4, 1.4) ; 
 \draw[-] (3.4, 1.4) -- (3.4, -1.4) ; 
 \draw[-] (3.4, -1.4) -- (0.6, -1.4) ;
 \draw (2,1.5) node[above] {$1$};
 \draw (2,-1.5) node[below] {$3$};
 \draw (0.5,0) node[left] {$2$};
 \draw (3.5,0) node[right] {$4$};
\end{tikzpicture}
\caption{Drawings of $\alpha$ (left), and $\alpha '$ (right)} 
\label{P2}
\end{figure}
\end{example}

Thus it is natural to consider the dihedral group $D_{n}$ (viewed as a subgroup of $\frak{S}_n$)  acts on the set of quasi-diagrams by conjugation:
\begin{align*}
D_{n} \times \frak{D}_n &\to \frak{D}_n,\\
(g, \alpha)&\mapsto g\cdot \alpha = g \alpha g^{-1}.
\end{align*}
We aim to answer the Question 2 mentioned  in the introduction.

Firstly, we consider the reflection $\gamma \in D_n$ which interchanges the sides $1$ and $n$, say
\[
\gamma= \begin{cases}
	(1,2m+1)(2,2m)\cdots (m, m+2)& \text{if}\  n=2m+1;\\
	(1,2m)(2, 2m-1)\cdots (m,m+1)& \text{if}\  n=2m.
\end{cases}
\]

\begin{proposition} \label{prop-ref}
Let $\alpha \in \frak{D}_n$ be a quasi-diagram. Then $\alpha$ is regular if and only if the quasi-diagram $\gamma \cdot \alpha$ is regular. 
\end{proposition}

\begin{proof}
	It is direct to check that $\gamma \zeta^{-1}\gamma=\zeta$. Then $$(\zeta(\gamma\cdot\alpha))^{-1} = (\zeta\gamma\alpha\gamma)^{-1} = \gamma\alpha\gamma\zeta^{-1}=\gamma(\alpha\gamma\zeta^{-1}\gamma)\gamma^{-1}=\gamma(\alpha\zeta)\gamma^{-1}.$$

Clearly, a subset $\omega$ is an orbit of $\alpha\zeta$ if and only if $\gamma(\omega)=\{\gamma(i)\mid i\in \omega\}$ is an orbit of $(\zeta(\gamma\cdot\alpha))^{-1}$, if and only if $\gamma(\omega)$ is an orbit of $\zeta(\gamma\cdot\alpha)$. Moreover, $n\in \omega$ if and only if $1\in \gamma(\omega)$, and $i\in \omega$ is an isolated point of $\alpha$ if and only if $\gamma(i)$ is an isolated point of $\gamma\alpha\gamma=\gamma\cdot \alpha$. Thus  $\omega\in \mathscr{B}_\alpha$ if and only if $\gamma(\omega)\in \mathscr{A}_{\gamma\cdot\alpha}$,
and the assertion follows from Theorem \ref{thm-main}.
\end{proof}

\begin{remark}
Let $\alpha$ be a quasi-diagram. 
Let $A$ and $A'$ be gentle algebras associated to $\alpha$ and $\gamma \cdot \alpha$ respectively. Then $A'$ is isomorphic to $A^{op}$, the opposite algebra of $A$.
\end{remark}

For a quasi-diagram $\alpha \in \frak{D}_n$ and an integer $l$. We call $\zeta^l \cdot \alpha = \zeta^{l} \alpha \zeta^{-l}$ the {\it $l$-th rotation} of $\alpha$. A quasi-diagram is called {\it rotatably regular} if its $l$-th rotation is regular for any integer $l$. 
Set
$$
\frak{R}_n : = \{ \text{rotatably regular quasi-diagrams} \in \frak{D}_n\}.
$$ 
Since $D_n$ is generated by $\gamma, \zeta$, and $\gamma \zeta = \zeta^{-1}\gamma$, Proposition \ref{prop-ref} implies that conjugate action of the dihedral group  on $\mathfrak S_n$ restricts to an action of $D_n$ on $\frak{R}_n$:
\begin{align*}
D_{n} \times \frak{R}_n \to \frak{R}_n,\quad
(g, \alpha)\mapsto g\cdot \alpha = g \alpha g^{-1}.
\end{align*}

\begin{lemma} \label{lem-rot}
Let $\alpha \in \frak{D}_n$ be a quasi-diagram, and $\beta_{l}$ the $l$-th rotation of $\alpha$ where $l \in \mathbb{Z}$. Then
\begin{itemize}
\item[(1)] $i$ is an isolated point of $\alpha$ if and only if $\zeta^l(i)$ is an isolated point of $\beta_l$;
\item[(2)] $w$ is a face of $\alpha$ if and only if $\zeta^l w \zeta^{-l}$ is a face of $\beta_l$. 
\end{itemize}
\end{lemma}

\begin{proof}
For (1). This follows from 
$\alpha(i) = i$ if and only if $ \zeta^l \alpha(i) = \zeta^l(i)$, and 
$
\beta_l \zeta^l (i)= \zeta^l \alpha \zeta^{-l} \zeta^l (i) = \zeta^l \alpha(i).
$

For (2). This follows from $\zeta\beta_l= \zeta(\zeta^l\alpha\zeta^{-l}) =\zeta^{l} (\zeta\alpha) \zeta^{-l}$. 
\end{proof}

We characterize the rotatably regular quasi-diagrams in the following. 

\begin{theorem} \label{thm-rot}
Let $\alpha \in \frak{D}_n$ be a  quasi-diagram. Then $\alpha$ is rotatably regular if and only if either $\alpha$ is maximal, or any face of $\alpha$ contains at least one isolated point.
\end{theorem}

\begin{proof}
Let $\alpha$ be a quasi-diagram, and $\beta_l = \zeta^l \cdot \alpha = \zeta^{l} \alpha \zeta^{-l}$ the $l$-th rotation of $\alpha$ where $l \in \mathbb{Z}$. 

First we prove the sufficiency. If $\alpha$ is maximal, say $\alpha$ has only one face, then any $\beta_l$ is maximal by Lemma \ref{lem-rot} (2), and hence regular; if any face of $\alpha$ contains at least one isolated point, then so is $\beta_l$ for any $l$ by Lemma \ref{lem-rot}, and hence $\beta_l$ is regular for any $l$.

Now assume that $\alpha$ is rotatably regular. To prove the necessity, it suffices to show that if there is a face $\omega$ of $\alpha$ such that $1\in\omega$ and $\omega$ contains no isolated points, then $\alpha$ is maximal.

By Lemma \ref{lem-rot}, $\zeta^l\omega\zeta^{-l}$ is a face of $\beta_l$ and contains no isolated points for any $l$. Since  $\alpha$ is a rotatably regular, $\beta_l$ is regular, and by Theorem  \ref{thm-main} it forces that $1\in \zeta^l\omega\zeta^{-l}$, or equivalently, $\zeta^{-l}(1)=n+1-l\in\omega$ for all $l$. Hence $\omega$ contains all points $i \in \{1,2, \cdots, n \}$, which means $\alpha$ is maximal. 
\end{proof}

\begin{example}
The first nontrivial set of rotatably regular quasi-diagrams is 
$$
\frak{R}_4 = \{\id, (13), (24), (13)(24) \}.
$$
Where
$$
D_4 \cdot (13) = \{ (13), (24) \}=\ D_4 \cdot (24), \ D_4 \cdot (13)(24) = \{ (13)(24) \}.
$$
\end{example}

\section{Expansions and Contractions} \label{sec-ptext}

We may identify $\mathfrak S_n$ with the subgroup of $\mathfrak S_{n+1}$ which fixes the point $n+1$. In this sense, any quasi-diagram $\alpha \in \frak{D}_n$ can be viewed as a diagram in $\frak{D}_{n+1}$. Moreover, for any $1\le i\le n+1$, we define
$$
\iota_i (\alpha) := \vartheta_{i,n+1} \alpha (\vartheta_{i,n+1})^{-1} \in \frak{D}_{n+1},
$$
where $\vartheta_{i,n+1} = (i, i+1, \dots, n, n+1) \in \frak{S}_{n+1}$.
We call $\iota_i (\alpha)$ is the {\it $i$-expansion} of $\alpha$, or the expansion of $\alpha$ at the position $i$. 

\begin{example}
	Let $\alpha = (13)(24) \in \frak{D}_4$ be a regular quasi-diagram. Then we have
	$$
	\iota_1(\alpha) = (12345)(13)(24)(54321) = (24)(35) \in \frak{D}_5,
	$$
	and 
	$$
	\iota_2(\alpha) = (2345)(13)(24)(5432) = (14)(35) \in \frak{D}_5.
	$$
	We can check that both $\iota_1(\alpha)$ and $\iota_2(\alpha)$ are regular. 
\end{example}

\begin{example} \label{exm-ptext2}
	Let $\alpha = (12) \in \frak{D}_2$ be a quasi-diagram which is not regular. Then we have
	$$
	\iota_1(\alpha) = (123)(12)(321) = (23) \in \frak{D}_3,
	$$
	and 
	$$
	\iota_2(\alpha) = (23)(12)(32) = (13) \in \frak{D}_3.
	$$
	We can check that $\iota_1(\alpha)$ is not regular, but $\iota_2(\alpha)$ is regular. 
\end{example}

Let $\alpha \in \frak{D}_n$ be a quasi-diagram. We may write
$$
\zeta_n \alpha = w_1 \cdots w_t
$$
as a complete product of disjoint cycles, or in other words,  $w_s$'s are all faces of $\alpha$. 
Let $i \in \{1, \dots, n+1 \}$. 
We consider the faces of $\iota_i (\alpha) \in \frak{D}_{n+1}$.

\begin{lemma} \label{lem-facext}
	Keep notations as above. 
	\begin{itemize}
		\item[(1)] If $i = n+1$, then $\zeta_{n+1}\iota_i (\alpha)=w_1'\cdots w_t'$, where
		$$
		w'_s = \begin{cases}  w_{s} & \text{\quad if } 1 \notin w_{s}, \\
			(1,n+1) w_{s} & \text{\quad otherwise.} \end{cases}
		$$
		\item[(2)] If $i \neq n+1$, then $\zeta_{n+1}\iota_i (\alpha)=w_1'\cdots w_t'$, where
		$$
		w'_s = \begin{cases} \vartheta_{i,n+1} w_{s} (\vartheta_{i,n+1})^{-1} & \text{\quad if } i \notin w_{s}, \\
			\vartheta_{i+1,n+1} w_{s} (\vartheta_{i,n+1})^{-1} & \text{\quad otherwise.} \end{cases}
		$$
	\end{itemize}
\end{lemma}

\begin{proof}
	Keep notations as mentioned before the lemma. The case $(1)$ is easy to check. We give a detailed proof of case $(2)$. 
	Assume $i \neq n+1$. We have 
	\begin{align*}
		\zeta_{n+1} \iota_i(\alpha) &= \zeta_{n+1} \vartheta_{i,n+1}\alpha (\vartheta_{i,n+1})^{-1} \\
		&= \zeta_{n+1}(i,i+1,\dots,n+1)\zeta_{n+1}^{-1}\zeta_{n+1}\alpha(\vartheta_{i,n+1})^{-1} \\
		&=(i+1,i+2,\dots,n+1,1)\zeta_{n+1}\alpha (\vartheta_{i,n+1})^{-1}\\
		&=(i+1,i+2,\dots,n+1,1)(1,n+1)\zeta_n \alpha (\vartheta_{i,n+1})^{-1}\\
		&=(i+1,i+2,\dots,n+1)\zeta_n \alpha (\vartheta_{i,n+1})^{-1} \\
		&=(i,i+1) (i, i+1, \dots,n+1) \zeta_n \alpha (\vartheta_{i,n+1})^{-1}\\
		&=(i,i+1) \vartheta_{i,n+1}w_1 \cdots w_t (\vartheta_{i,n+1})^{-1}. 
	\end{align*}
	Note there is only one $d \in \{1, \dots, t\}$ such that $i \in w_d$, so that 
	$$
	i+1 \in \vartheta_{i,n+1}w_d (\vartheta_{i,n+1})^{-1} =: \tilde{w}_d.
	$$
	Denote by $w'_s = \vartheta_{i,n+1}w_s (\vartheta_{i,n+1})^{-1}$ if $s \neq d$. Then we have
	\begin{align*}
		\zeta_{n+1} \iota_i(\alpha) &= (i,i+1) w'_1 \cdots w'_{d-1} \tilde{w}_d w'_{d+1} \cdots w'_t. 
	\end{align*}
	Since $w'_s$ does not contain $i$ nor $i+1$ for $s \neq d$, we have $(i,i+1) w'_s = w'_s (i,i+1)$. Thus
	\begin{align*}
		\zeta_{n+1} \iota_i(\alpha) &=  w'_1 \cdots w'_{d-1} (i,i+1) \tilde{w}_d w'_{d+1} \cdots w'_t. 
	\end{align*}
	Let $w'_d = (i,i+1) \tilde{w}_d$. Clearly $w_1', \cdots, w_t'$ are disjoint cycles, and the proof is completed  for $(i,i+1)\vartheta_{i,n+1} = \vartheta_{i+1,n+1}$. 
\end{proof}

\begin{remark}
	Let $\alpha$, $\iota_i(\alpha)$, $w_s$, $w_s'$ be as above.  Assume $w_s=(j_1\cdots j_r)$. Then $w_s'=(j_1'\cdots j_r')$ if $i\notin w_s$, where $j_u'= \vartheta_{i,n+1}(j_u)$, or more explicitly, $j_u'=j_u$ if $1\le j_u <i$, and $j_u'=j_u+1$ if $i\le j_u \le n$.
	
	If $i\in w_s$, we may assume $j_1=i$ without loss of generality. Then in this case, $w'_s= (i j_1'\cdots j_r')$, where again $j_u'=\vartheta_{i,n+1}(j_u)$. In particular, $j_1'=i+1$.
\end{remark}

\begin{lemma} \label{lem-isoext}
	Let $\alpha \in \frak{D}_n$, and $i \in \{1, \dots, n+1 \}$. If $\{ j_1, \dots, j_t\}$ is the set of isolated points of $\alpha$, then $\{ i, \vartheta_{i,n+1}(j_1), \cdots, \vartheta_{i,n+1}(j_t)\}$ is the set of isolated points of $\iota_i(\alpha) \in \frak{D}_{n+1}$.
\end{lemma}

\begin{proof}
	Let $\alpha \in \frak{D}_n$, and $i \in \{1, \dots, n+1\}$. Assume $\{ j_1, \dots, j_t\}$ is the set of isolated points of $\alpha$. Then $\alpha$ has $m = {\frac{n-t}{2}}$ chords 
	$
	\alpha_1, \dots, \alpha_{m}.  
	$
	Then
	\begin{align*}
		\iota_i(\alpha) =& \vartheta_{i,n+1} \alpha (\vartheta_{i,n+1})^{-1}  \\
		=& \vartheta_{i,n+1} \alpha_1 \cdots \alpha_m(j_1)\cdots(j_t) (n+1) (\vartheta_{i,n+1})^{-1}  \\
		=& (\vartheta_{i,n+1} \alpha_1 \cdots \alpha_m (\vartheta_{i,n+1})^{-1}) (\vartheta_{i,n+1}(j_1))\cdots(\vartheta_{i,n+1}(j_t)) (\vartheta_{i,n+1}(n+1)).
	\end{align*}
	Thus 
	$$
	\{\vartheta_{i,n+1}(j_1), \cdots, \vartheta_{i,n+1}(j_t), \vartheta_{i,n+1}(n+1) = i\}
	$$ 
	is the set of isolated points of $\iota_i(\alpha) \in \frak{D}_{n+1}$. 
\end{proof}

\begin{proposition} \label{prop-ptextreg}
	Assume that $\alpha \in \frak{D}_n$ is a regular quasi-diagram. Then the $i$-expansion $\iota_i(\alpha) \in \frak{D}_{n+1}$ is regular for all $i \in \{1, \dots, n+1 \}$.
\end{proposition}

\begin{proof}
	Write
	$
	\zeta_n \alpha = w_1 \cdots w_t,
	$
	where $w_s$'s are all faces of $\alpha$. By Theorem \ref{thm-main}, each $w_s$ contains $1$ or at least one isolated point of $\alpha$. If $i = n+1$, then it is easy to check that $\iota_{n+1}(\alpha)$ is also regular by Lemma \ref{lem-facext}. 
	
	Now we assume that $i \neq n+1$. Then there is $d \in \{ 1, \dots, t\}$ such that $i \in w_d$. By Lemma \ref{lem-facext}, $
	\zeta_{n+1} \iota_i(\alpha) = w'_1 \cdots w'_t,
	$ 
	where
	$$
	w'_s = \begin{cases} \vartheta_{i,n+1} w_{s} (\vartheta_{i,n+1})^{-1}, &  \text{\quad if } s \neq d, \\
		\vartheta_{i+1,n+1} w_{d} (\vartheta_{i,n+1})^{-1} & \text{\quad otherwise.} \end{cases}
	$$
	
	We will show that each  $w_s'$ contains either some isolated point of $\iota_i(\alpha)$ or $1$. This can be checked case by case.
	
	(1) $s\ne d$, and $w_s$ contains some isolated point $j$ of $\alpha$. Then $\vartheta_{i,n+1}(j)$ is an isolated point by Lemma \ref{lem-isoext}, and clearly $\vartheta_{i,n+1}(j) \in w'_s$. 
	
	(2) $s\ne d$ and $1\in w_s$. Then $i > 1$ by assumption, otherwise, $w_s = w_d$ which is a contradiction. Thus $1=\vartheta_{i,n+1}(1) \in w'_s$. 
	
	(3) $s=d$. Since $w_d$ contains $i$, then 
	$
	\tilde{w}_d = \vartheta_{i,n+1}w_d (\vartheta_{i,n+1})^{-1}
	$ 
	contains $i+1$ and does not contain $i$. Thus $w'_d = (i,i+1)\tilde{w}_d$ contains $i$ which is an isolated point of $\iota_i(\alpha)$ by Lemma \ref{lem-isoext}.
\end{proof}

The above result tells us that if a quasi-diagram $\alpha \in \frak{D}_n$ is regular, then $\alpha$ is also regular when viewed as a quasi-diagram in $\frak{D}_m$ for any $m>n$. 

\begin{proposition} \label{prop-ptextmax}
	Let $\alpha \in \frak{D}_n$ be a maximal quasi-diagram. Then $\iota_i(\alpha) \in \frak{D}_{n+1}$ is also a maximal quasi-diagram for all $i \in \{1, \dots, n+1 \}$.
\end{proposition}

\begin{proof}
	By definition $\alpha$ is maximal if and only if $\alpha$ has only one face. Since expansions preserves the number of faces, we know that each $\iota_i(\alpha)$ has one face, and hence is maximal.
\end{proof}

Dual to the notion of expansions, we may also introduce {\it contractions} of a quasi-diagram at isolated points. 
Let $\alpha \in \frak{D}_n$ be a quasi-diagram such that $\alpha$ is not a chord diagram, i.e., the set of isolated points of $\alpha$ is not empty. Assume $i \in \{1, \dots, n \}$ is an isolated point of $\alpha$. Then we define
$$
\delta_i(\alpha) : = (\vartheta_{i,n})^{-1} \alpha \vartheta_{i,n}. 
$$
As above, we may identify $\mathfrak S_{n-1}$ as the subgroup of $\mathfrak S_n$ which fixes $n$. Now $n$ is an isolated point of $\delta_i(\alpha)$, and we may view $\delta_i(\alpha)$ as a quasi-diagram in $\frak{D}_{n-1}$ by omitting the isolated point $n$. We call $\delta_i(\alpha)$  the {\it $i$-contraction} of $\alpha$, or the contraction of $\alpha$ at the position $i$. Note that we always have
$
\delta_i \iota_i = \id,
$
and 
$
\iota_i \delta_i  = \id
$
when the composition is well defined. 

The following two lemmas are the dual versions of Lemma \ref{lem-facext} and Lemma \ref{lem-isoext}, respectively. The proofs are easily obtained by applying the fact that $\iota_i\delta_i(\alpha)=\alpha$ and Lemma \ref{lem-facext} and  Lemma \ref{lem-isoext}, and we omit them here. 

\begin{lemma} \label{lem-facered}
	Let $\alpha \in \frak{D}_n$, and let $i \in \{1, \dots, n \}$ be an isolated point of $\alpha$. Assume
	$
	\zeta_n \alpha = w_1 \cdots w_t,
	$
	where $w_s$'s are faces of $\alpha$.
	\begin{itemize}
		\item[(1)]If $i = n$, then $\zeta_{n-1}\delta_i(\alpha)= w_1'\cdots w_t'$, where 
		$$
		w'_s = \begin{cases}  w_{s} & \text{\quad if } 1 \notin w_{s}, \\
			(1,n) w_{s} & \text{\quad otherwise.} \end{cases}
		$$
		\item[(2)]  If $i \neq n$, then  $\zeta_{n-1}\delta_i(\alpha)= w_1'\cdots w_t'$, where  
		$$
		w'_s = \begin{cases} (\vartheta_{i,n})^{-1} w_{s} \vartheta_{i,n} & \text{\quad if } i \notin w_{s}, \\
			(\vartheta_{i+1,n})^{-1} w_{s} \vartheta_{i,n} & \text{\quad otherwise.} \end{cases}
		$$
	\end{itemize}
\end{lemma}

\begin{lemma} \label{lem-isoext2}
	Let $\alpha \in \frak{D}_n$, and $i \in \{1, \dots, n \}$ be an isolated point of $\alpha$. If $\{ j_1, \dots, j_t\}$ is the set of isolated points of $\alpha$, then 
	$$
	\{(\vartheta_{i,n})^{-1}(j_1), \cdots, (\vartheta_{i,n})^{-1}(j_t)\} \setminus \{ n \}
	$$
	is the set of isolated points of $\delta_i(\alpha) \in\frak{D}_{n-1}$.
\end{lemma}

By Proposition \ref{prop-ptextreg}, an expansion of a regular quasi-diagram remains regular, but the regularity is not preserved under taking contractions. For instance,  $(13) \in \frak{D}_3$ is regular while its $2$-contraction  $(12) \in \frak{D}_2$ is not regular (see Example \ref{exm-ptext2}). 
However, the converse version of Proposition \ref{prop-ptextmax} is true since an contraction preserves the number of faces. 

\begin{proposition} \label{prop-ptredmax}
	Let $\alpha \in \frak{D}_n$ be a maximal quasi-diagram and $i \in \{1, \dots, n \}$ be an isolated point of $\alpha$. Then $\delta_i(\alpha) \in \frak{D}_{n-1}$ is maximal. 
\end{proposition}

Let $\alpha \in \frak{D}_n$, and $\alpha' \in \frak{D}_{n+r}$ be quasi-diagrams. We say $\alpha'$ is an iterated expansion of $\alpha$ if 
$$
\iota_{i_r}\cdots \iota_{i_2}\iota_{i_1} (\alpha) = \alpha' \in \frak{D}_{n+r}.
$$  
Which is equivalent to say
$$
\alpha = \delta_{i_1} \cdots \delta_{i_{r-1}}\delta_{i_r} (\alpha') \in \frak{D}_n.
$$

Recall that $\id \in \frak{D}_n$ is called the trivial quasi-diagram. In summary, we have the following result which connects (maximal) chord diagrams and (maximal) quasi-diagrams.

\begin{proposition} \label{prop-mdisptext}
	Every nontrivial quasi-diagram is an iterated expansion of a chord diagram, and every nontrivial maximal quasi-diagram is an iterated expansion of a maximal chord diagram. 
\end{proposition}

\begin{remark}
	The result above is mentioned in \cite[Section 6]{CM} implicitly. 
\end{remark}

\section{Maximal Chord Diagrams} \label{sec-cd}

In this section, we discuss the set $\mathfrak M_{n}$ of maximal chord diagrams. We may assume $n=2m$ for some positive integer $m$. Note that chord diagrams exist only when $n$ is an even number.

There is a natural oriented surface without boundary associated to a chord diagram \cite{CM,K}.
Let $P_{2m}$ be a $2m$-gon with its sides labeled by $1,2, \dots, 2m$ consecutively around its boundary. 
A chord diagram $\alpha \in \frak{D}_{2m}$ defines a way to glue the sides $i, \alpha(i)$ of $P_{2m}$ so that yields an oriented surface $\mathscr{S}_\alpha$ with an one-face map formed by edges and vertices of the polygon. The number of vertices in $\mathscr{S}_\alpha$ equals to the number of faces of $\alpha$. Thus the formula of genus of $\mathscr{S}_\alpha$ is same to Equation \ref{equ-gen}, which implies that  $\mathscr{S}_\alpha$ has maximal possible genus if $\alpha$ is maximal. 

\begin{example}
The maximal chord diagram $\alpha = (13)(24) \in \frak{D}_4$ defines a way to glue $P_4$ to a torus with $2$ edges and one vertex (Figure \ref{p4}). 
\begin{figure}[h]
\centering
\begin{tikzpicture}[x=20, y=20, baseline=(current bounding box)]
\draw[magenta,-] (-1.4, -1.4) --  (-1.4,1.4) ; 
\draw[magenta,<<-] (-1.4, 0) --  (-1.4,1.4) ; 
\draw[green,-] (-1.4,1.4) -- (1.4, 1.4) ; 
\draw[green,<-] (0,1.4) -- (1.4, 1.4) ; 
\draw[magenta,-] (1.4, 1.4) -- (1.4, -1.4) ; 
\draw[magenta,->>] (1.4, 1.4) -- (1.4, 0) ; 
\draw[green,-] (1.4, -1.4) -- (-1.4, -1.4) ; 
\draw[green,->] (1.4, -1.4) -- (0, -1.4) ; 

\draw (0,1.5) node[above] {$1$};
\draw (0,-1.5) node[below] {$3$};
\draw (-1.5,0) node[left] {$2$};
\draw (1.5,0) node[right] {$4$};

\draw[fill=white] (7.5,0) ellipse (2.4 and 1.35) 
 (7,0) arc(120:60:1 and 1.25) arc(-60:-120:1 and 1.25);
\draw (7,0) arc(-120:-130:1 and 1.25) (8,0) arc(-60:-50:1 and 1.25);
\draw[green] (7.5,0.1) ellipse (1.2cm and 0.55cm);
\draw[magenta] (7.55,-0.163) arc (68:-42:0.735);
\fill (7.985,-0.65) circle (0.8pt) ;
\end{tikzpicture}
\caption{Glue $P_2$ by $\alpha=(13)(24)$} 
\label{p4}
\end{figure}
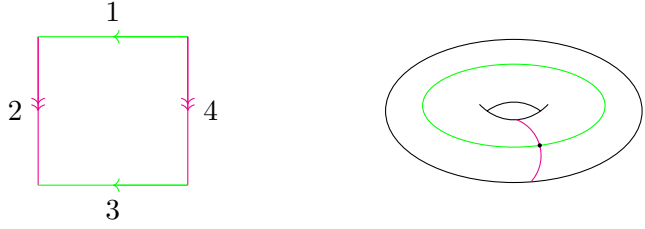
\end{example}

\begin{remark}
Let $\mathbb{S}_{A}$ be the marked surface of the gentle algebra $A$ associated to a chord diagram $\alpha$. Then $\widehat{\mathbb{S}_{A}}$ is an oriented surface which is dual to $\mathscr{S}_\alpha$
 by interchanging faces and vertices. 
If $\alpha$ is maximal, then $\mathscr{S}_\alpha$ and $\widehat{\mathbb{S}_{A}}$ both have one vertex and one face. On the other hand, since both edges in $\mathscr{S}_\alpha$ and $\widehat{\mathbb{S}_{A}}$ correspond to the chords of $\alpha$ with the same order counterclockwisely (or clockwisely) around the vertex in the surface, they are homeomorphic. 
\end{remark}

We summarize the properties of maximal chord diagrams as follows. 
   
\begin{proposition} \label{prop-mcd}
Let $\alpha \in \frak{D}_{2m}$ be a chord diagram, and $A$ the associated gentle algebra. Then the following statements are equivalent.
\begin{itemize}
	\item[(1)] $\alpha$ is maximal.
	\item[(2)] $\alpha$ is regular. 
	\item[(3)]  $\alpha $ is rotatably regular.
	\item[(4)] The $l$-th rotation of $\alpha$ is maximal for any integer $l$. 
	\item[(5)] The Koszul dual quasi-diagram $\alpha^!$ exists.
	\item[(6)] $\alpha$ is the Koszul dual of a maximal chord diagram.
	\item[(7)] $\gldim A = 2m-1$. 
\end{itemize}
Moreover, if $\alpha$ is maximal, then $m$ is an even number.
\end{proposition}

\begin{proof}
{ $(1) \Leftrightarrow (2)$}  follows from Theorem \ref{thm-main}, since a chord diagram is by definition isolated point free. 

{ $(1) \Leftrightarrow (3) \Leftrightarrow (4)$}  follows from Lemma \ref{lem-rot} and Theorem \ref{thm-rot}.

{ $(1) \Leftrightarrow (5) \Leftrightarrow (6) \Leftrightarrow (7)$}  follows from Theorem \ref{thm-kosdual}.

Moreover, if $\alpha$ is maximal, then $\zeta_{2m} \alpha = \eta$ for some  $2m$-cycle $\eta$. It forces that $\alpha$ is an even permutation and hence $m$ is an even number. 
 \end{proof}

\begin{remark}
	Let $A$  be a gentle algebra defined by a maximal chord diagram $\alpha \in \frak{M}_{2m}$, by Equation \ref{equ-gen}, the genus of the marked surface of $A$ is 
	$$
	g = \frac{1}{2}(\#\{\text{chords of } \alpha\}  - \# \{\text{faces of  } \alpha \}  +1) = \frac{1}{2}(m - 1 +1) = \frac{m}{2}.
	$$
Which also explains that $m$ has to be an even number. The above equation also suggests us that one should consider ${4g}$ rather than ${2m}$. 
\end{remark}

We need some further notions. Let $n$ be a positive integer and $E \subset \frak{D}_{n}$. We say $E$ is {\it closed under taking rotations} if the conjugate action
$$
\langle \zeta_n \rangle \times E \to E
$$
is well defined, where $\langle \zeta_n \rangle$ is the subgroup of $D_n$ generated by $\zeta_n$. We say $E$ is {\it closed under taking Koszul dual} if the map
$$
(-)^!: E \to E, \ \alpha \mapsto \alpha^!
$$
is well defined.

As we have shown in Proposition \ref{prop-mcd}, the set $\frak{M}_{4g} \subset \frak{D}_{4g}$ of maximal chord diagram is 
\begin{equation}
	\tag{$\ast$}
	\text{closed under taking rotations and Koszul dual. }
\end{equation}
Next we will show that the converse is also true, say, any nonempty subset of quasi-diagrams satisfies the condition  $(*)$ is contained in some $\frak{M}_{4g}$. 
The only exceptional cases are $n=1,2$.
\begin{example}
	Let $E= \{ \id \} \subset \frak{D}_2$. Then $E$ is closed under taking Koszul dual (Example \ref{exm-id2}), and closed under taking rotations (Theorem \ref{thm-rot}). 
\end{example}

\begin{proposition} \label{prop-setmcd}
Assume that $n \geq 3$ and $E \subset \frak{D}_{n}$ is a nonempty subset satisfying condition $(*)$. Then $n = 4g$ for some integer $g$, and $E \subset \frak{M}_{4g}$.
\end{proposition}
	
\begin{proof}
Assume $E \subset \frak{D}_{n}$ satisfies the condition $(*)$. Let $\alpha \in E$. We want to show that $\alpha$ is a maximal chord diagram. 

We claim that $\alpha$ is rotatably regular. Since $E$ is closed under taking Koszul dual by assumption. By Theorem \ref{thm-kosdual}, $\alpha \in E$ is regular. Since $E$ is closed under taking rotation by assumption, we have $\zeta^l \cdot \alpha \in E$ for any $l \in \mathbb{Z}$, thus $\zeta^l \cdot \alpha$ is regular for any $l \in \mathbb{Z}$. Thus $\alpha$ is rotatably regular. 

Then by Theorem \ref{thm-rot} and Theorem \ref{thm-kosdual}, $\alpha$ is one of the following cases
\begin{itemize}
\item[(i)] $\alpha$ is a maximal chord diagram. 
\item[(ii)] $\alpha$ is a maximal quasi-diagram with isolated points $1$ or $n$, or both. 
\end{itemize}
If $\alpha$ is a maximal chord diagram, then we are done. 
Now we assume $\alpha$ is a maximal quasi-diagram with an isolated point $i$, where $i\in \{1,n \}$.
Let $\beta_l = \zeta_n^l \cdot \alpha \in E$. Then by the proof of Theorem \ref{thm-rot}, 
$\beta_l$ is a maximal quasi-diagram with isolated point $\zeta_n^l(i)$. 
Note that $n \geq 3$. If $i =1$ (resp. $i = n$), 
then $\zeta_n(i) \neq 1,n$ (resp. $\zeta_n^2(i) \neq 1,n$). So the Koszul dual of $\beta_1$ (resp. $\beta_2$) does not exists by Theorem \ref{thm-kosdual}, which is a contradiction. 
\end{proof}

\begin{remark} \label{rem-tg-mc}
Chang and Schroll showed that for a gentle algebra $A$ with finite global dimension, 
the derived category $D^b(\mod A)$ has a full exceptional sequence if and only if 
the marked surface $\mathbb{S}_{A}$ of $A$ is not homeomorphic to $\mathbb{T}_{g,1,1}$, 
where $\mathbb{T}_{g,1,1}$ is an oriented surface of genus $g \geq 1$, with only one boundary 
component and only one marked point (\cite[Theorem 3.7]{CS}). By \cite[Lemma 3.3]{CS}, the class of gentle algebras of finite global dimension associated to $\mathbb{T}_{g,1,1}$ is just the class of gentle algebras associated to maximal chord diagrams $\frak{M}_{4g}$. 
\end{remark}

We have the following counting formulae of maximal chord diagrams and maximal quasi-diagrams.

\begin{proposition} \label{prop-enum} Let $g, n$ be positive integers, and assume $n=4t + q$, where $t, q$ are integers with $0\le q\le 3$. Then
\begin{itemize}
\item[(1)] {\rm (\cite[Equation 14]{WL}, see also \cite[Theorem 2]{HZ})}
$$
\#(\frak{M}_{4g})= \varepsilon_g :=  \frac{(4g)!}{4^g(2g+1)!}; 
$$
\item[(2)] {\rm (\cite[Section 6]{CM})} 
$$
\#\{\text{maximal quasi-diagrams} \in \frak{D}_n \} = \sum_{i=0}^t \begin{pmatrix} 4t + q \\ 4i \end{pmatrix} \varepsilon_i
$$
where we define $\varepsilon_0=1$.
\end{itemize} 
\end{proposition}

We mention that  there is some misprint in the original version of Proposition \ref{prop-enum} (2) appeared in \cite[Section 6]{CM}. 
Here we give a modified version, which is an easy consequence of Proposition \ref{prop-mdisptext} and we omit the proof.

Moreover, Cori and Marcus gave a counting formula of equivalence classes of maximal chord diagrams up to rotation \cite{CM}, and 
Krasko gave a counting formula of orbits of maximal chord diagrams $\frak{M}_{4g}$ under the action of the dihedral group $D_{4g}$ \cite{K}. See the example $\frak{M}_8$ in the appendix.

\section{Appendix. The Case $\frak{M}_8$}

There are $21$ different maximal chord diagrams in $\frak{M}_8$, but
only $4$ types of orbits for $\frak{M}_8$ under the group action of $D_8$. 
We list these four types of orbits in Figure \ref{type1}, Figure \ref{type2}, Figure \ref{type3}, Figure \ref{type4}.
In these figures, $\zeta, \gamma \in D_8$ are the rotation and reflection act on $\frak{M}_8$ defined in Section \ref{sec-dih}, and $(-)^!$ maps $\alpha \in \frak{M}_8$ to its Koszul dual $\alpha^!$. 

\begin{figure}[h]
\centering
\begin{tikzpicture}[commutative diagrams/every diagram]
\node (a1) at (0,0)  {\small$(15)(26)(37)(48)$};

\draw (1.4,0.3) edge[out=65, in=5, loop, distance=2.5em, cyan,<->] node[right] {\tiny{{$\gamma\cdot$}}} (1.63,-0.1); 
\draw (-1.4,0.3) edge[out=115, in=-180, loop, distance=2.5em, orange,<->] node[left] {\tiny{{$(-)^!$}}} (-1.63,-0.1); 

\path[commutative diagrams/.cd, every arrow, every label]
(a1) edge[out=-60, in=-120, loop, distance=2.5em,<->] node {$\zeta \cdot$} (a1);
\end{tikzpicture}
\caption{Type I} 
\label{type1}
\end{figure}

\begin{figure}[h]
\centering
\begin{tikzpicture}[commutative diagrams/every diagram]
\node (a1) at (-3.5,3)  {\small$(13)(24)(57)(68)$};
\node (a2) at (3.5,3)  {\small$(17)(28)(35)(46)$};
\node (a3) at (0,1.5)  {\small$(13)(28)(46)(57)$};
\node (a4) at (0,0)  {\small$(17)(24)(35)(68)$};

\draw (1.4+3.5,0.3+3) edge[out=65, in=5, loop, distance=2.5em, cyan,<->] node[right] {\tiny{{$\gamma\cdot$}}} (1.63+3.5,-0.1+3); 
\draw (-1.4-3.5,0.3+3) edge[out=115, in=-180, loop, distance=2.5em, cyan,<->] node[left] {\tiny{{$\gamma\cdot$}}} (-1.63-3.5,-0.1+3); 
\draw (1.4,0.3+1.5) edge[out=65, in=5, loop, distance=2.5em, orange,<->] node[right] {\tiny{{$(-)^!$}}} (1.63,-0.1+1.5); 

\path[commutative diagrams/.cd, every arrow, every label]
(a1) edge[bend right=40] node[swap] {$\zeta\cdot$} (a4)
(a4) edge[bend right=40] node[swap] {$\zeta\cdot$} (a2)
(a3) edge[bend right=20] node {$\zeta\cdot$} (a1)
(a3) edge[bend left=20,<-] node {$\zeta\cdot$} (a2)
(a1) edge[out=60, in=130, loop, orange, distance=2.5em,<->] node[swap] {$(-)^!$} (a1)
(a2) edge[out=60, in=130, loop, orange, distance=2.5em,<->] node[swap] {$(-)^!$} (a2)
(a3) edge[cyan,<->] node {$\gamma\cdot$} (a4)
(a4) edge[out=-60, in=-120, loop, distance=2.5em, orange,<->] node {$(-)^!$} (a4);
\end{tikzpicture}
\caption{Type II} 
\label{type2}
\end{figure}

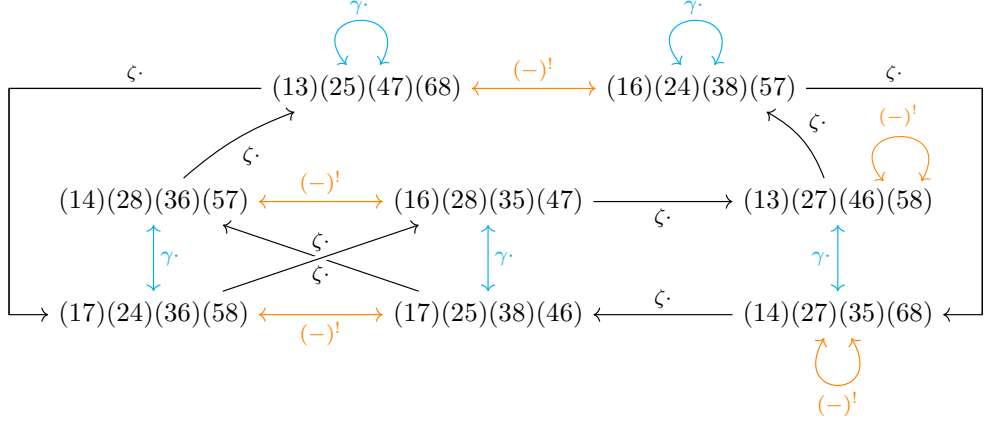
\begin{figure}[h]
\centering
\begin{tikzpicture}[commutative diagrams/every diagram]
\node (a1) at (-2.2,3)  {\small$(13)(25)(47)(68) $};
\node (a2) at (2.2,3)  {\small$(16)(24)(38)(57)$};
\node (a3) at (-5,1.5)  {\small$(14)(28)(36)(57)$};
\node (a4) at (-0.6,1.5)  {\small$ (16)(28)(35)(47)$};
\node (a5) at (-5,0)  {\small$(17)(24)(36)(58)$};
\node (a6) at (-0.6,0) {\small$ (17)(25)(38)(46)$};
\node (a7) at (4,1.5) {\small$(13)(27)(46)(58)$};
\node (a8) at (4,0) {\small$(14)(27)(35)(68)$};



\draw (5.1,1.77) edge[out=60, in=130, loop, orange, distance=2.5em,<->] node[above] {{\tiny$(-)^!$}} (4.6,1.77);

\path[commutative diagrams/.cd, every arrow, every label]
(a1) edge[out=60, in=130, loop, cyan, distance=2.5em,<->] node[swap] {$\gamma\cdot$} (a1)
(a1) edge[orange,<->] node {$(-)^!$} (a2)
(a1) edge[bend right=10,<-] node {$\zeta\cdot$} (a3)

(a2) edge[bend left=20,<-] node {$\zeta\cdot$} (a7)
(a2) edge[out=60, in=130, loop, cyan, distance=2.5em,<->] node[swap] {$\gamma\cdot$} (a2)

(a3) edge[cyan, <->] node {$\gamma\cdot$} (a5)
(a3) edge[<-] node[above] {$\zeta\cdot$} (a6)
(a3) edge[orange,<->] node {$(-)^!$} (a4)
(a4) edge[white, line width=1mm, -]  (a5)
(a4) edge[<-] node[below] {$\zeta\cdot$} (a5)
(a4) edge[cyan, <->] node {$\gamma\cdot$} (a6)
(a6) edge[,<-] node {$\zeta\cdot$} (a8)
(a5) edge[orange,<->] node[swap] {$(-)^!$} (a6)

(a7) edge[<-] node {$\zeta\cdot$} (a4)
(a7) edge[cyan,<->] node[swap] {$\gamma\cdot$} (a8)
(a8) edge[out=-60, in=-120, loop, distance=2.5em, orange,<->] node {$(-)^!$} (a8)

(a2) edge[-] node[above]{{\small $\zeta\cdot$}}  (5.9,3)
(a8) edge[<-,to path=-| (\tikztotarget)] (5.9,3) 
(a2) edge[-, to path=-| (\tikztotarget)] (5.9,0)

(a1) edge[-] node[above]{{\small $\zeta\cdot$}}  (-6.9, 3)
(a1) edge[-,to path=-| (\tikztotarget)] (-6.9,0) 
(a5) edge[<-, to path=-| (\tikztotarget)] (-6.9,3)
;
\end{tikzpicture}
\caption{Type III} 
\label{type3}
\end{figure}

\begin{figure}[h]
\centering
\begin{tikzpicture}[commutative diagrams/every diagram]
\node (a1) at (-2.2,3)  {\small$(13)(26)(47)(58)$};
\node (a2) at (2.2,3)  {\small$(16)(27)(35)(48)$};
\node (a3) at (-2.2,1.5)  {\small$(14)(25)(37)(68)$};
\node (a4) at (2.2,1.5)  {\small$(15)(27)(38)(46)$};
\node (a5) at (-5,0)  {\small$(15)(28)(36)(47)$};
\node (a6) at (-5,-1.5) {\small$(17)(25)(36)(48)$};
\node (a7) at (4,0) {\small$(14)(26)(38)(57)$};
\node (a8) at (4,-1.5) {\small$(16)(24)(37)(58)$};


\draw (-5.5,0.3) edge[out=60, in=130, loop, orange, distance=2.5em,<->] node[above] {{\tiny$(-)^!$}} (-6,0.3);
\draw (5.1,0.3) edge[out=60, in=130, loop, orange, distance=2.5em,<->] node[above] {{\tiny$(-)^!$}} (4.6,0.3);

\draw (1.4-5,0.3-1.5) edge[out=65, in=5, loop, distance=2.5em, orange,<->] node[right] {\tiny{{$(-)^!$}}} (1.63-5,-0.1-1.5); 
\draw (-1.4+4,0.3-1.5) edge[out=115, in=-180, loop, distance=2.5em, orange,<->] node[left] {\tiny{{$(-)^!$}}} (-1.63+4,-0.1-1.5); 
      
\path[commutative diagrams/.cd, every arrow, every label]
(a1) edge[orange,<->] node {$(-)^!$} (a2)
(a1) edge[cyan,<->] node {$\gamma\cdot$} (a3)

(a2) edge[cyan,<->,bend left=60] node {$\gamma\cdot$} (a4)
(a2) edge node[swap] {$\zeta\cdot$} (a4)
(a3) edge[bend left=15] node {$\zeta\cdot$} (a6)
(a3) edge[orange,<->] node {$(-)^!$} (a4)
(a4) edge[bend left=5] node {$\zeta\cdot$} (a7)

(a5) edge[bend left=20] node {$\zeta\cdot$} (a1)
(a6) edge node {$\zeta\cdot$} (a5)
(a6) edge[cyan, bend left=60,<->] node {$\gamma\cdot$} (a5)
(a7) edge[bend left=10] node {$\zeta\cdot$} (a3)
(a7) edge[cyan,<->] node {$\gamma\cdot$} (a8)
(a1) edge[-] node[above]{{\small $\zeta\cdot$}}  (-6.9, 3)
(a1) edge[-,to path=-| (\tikztotarget)] (-6.9,-2.10766) 
(a8) edge[<-, to path=|- (\tikztotarget)] (-6.9,-2.1)

(a2) edge[<-] node[above]{{\small $\zeta\cdot$}}  (5.9,3)
(a2) edge[-,to path=-| (\tikztotarget)] (5.9,-1.5) 
(a8) edge[-, to path=-| (\tikztotarget)] (5.9,3)
;
\end{tikzpicture}
\caption{Type IV} 
\label{type4}
\end{figure}




\end{document}